%% file: main.tex
\DeclareSymbolFont{AMSb}{U}{msb}{m}{n}
\documentclass[11pt,a4paper,twoside,leqno,noamsfonts]{amsart}
\usepackage[english]{babel}
\usepackage[dvipsnames]{xcolor}
\definecolor{britishracinggreen}{rgb}{0.0, 0.26, 0.15}
\definecolor{cobalt}{rgb}{0.0, 0.28, 0.67}
\usepackage[utopia]{mathdesign}
    \DeclareSymbolFont{usualmathcal}{OMS}{cmsy}{m}{n}
    \DeclareSymbolFontAlphabet{\mathcal}{usualmathcal}
\usepackage[a4paper,top=3cm,bottom=3cm,left=2.8cm,
           right=2.8cm,bindingoffset=5mm]{geometry}
\usepackage[utf8]{inputenc}
\usepackage{braket,caption,comment,mathtools,stmaryrd}
\usepackage{multirow,booktabs,microtype,relsize}
\usepackage{todonotes}
\usepackage{soul} 
\usepackage[colorlinks,bookmarks]{hyperref} %
\hypersetup{colorlinks,%
            citecolor=britishracinggreen,%
            filecolor=black,%
            linkcolor=cobalt,%
            urlcolor=black}
\numberwithin{equation}{section}
\input{macros}

\title[On the motive of the Quot scheme]{On the motive of the Quot scheme of finite quotients \\ of a locally free sheaf}

\author[A. T. Ricolfi]{Andrea T. Ricolfi}
\address{SISSA, Via Bonomea 265 Trieste}
\email[Andrea T. Ricolfi]{aricolfi@sissa.it}


\begin{document}
\maketitle

\begin{abstract}
Let $X$ be a smooth variety, $E$ a locally free sheaf on $X$. We express the generating function of the motives $[\Quot_X(E,n)]$ in terms of the power structure on the Grothendieck ring of varieties. This extends a recent result of Bagnarol, Fantechi and Perroni for curves, and a result of Gusein-Zade, Luengo and Melle-Hern\'{a}ndez for Hilbert schemes. 
We compute this generating function for curves and we express the relative motive $[\Quot_{\A^d}(\O^{\oplus r}) \to \Sym \A^d]$ as a plethystic exponential.
\end{abstract}

{\hypersetup{linkcolor=black}
\tableofcontents}

\section{Introduction}

Let $X$ be a smooth quasi-projective variety over $\C$, and let $E$ be a locally free sheaf of rank $r$ on $X$. The Quot scheme $\Quot_X(E,n)$ parameterises quotients $E\onto Q$ such that $Q$ is a zero-dimensional sheaf of length $n$. Recently Bagnarol, Fantechi and Perroni \cite{BFP19} have shown that if $C$ is a smooth proper curve, the class
\[
\bigl[\Quot_C(E,n) \bigr] \,\in\, K_0(\Var_{\C})
\]
in the Grothendieck ring of varieties does not depend on $E$. We use the theory of \emph{power structures} \cite{GLMps} to extend their result to arbitrary dimension. Roughly speaking, a power structure on a ring $R$ is a way of making sense of expressions $A(t)^m$, where $A(t) = 1+A_1t+A_2t^2+\cdots$ is a power series with coefficients in $R$ and $m \in R$.

For $(X,E)$ as above, we form the generating function
\[
\mathsf Z_E(t) = \sum_{n \geq 0}\,\bigl[\Quot_X(E,n) \bigr]t^n,
\]
and we denote by $\mathsf P_{r,n}\in K_0(\Var_{\C})$ the motive of the \emph{punctual Quot scheme}, namely the closed subscheme $P_{r,n} \subset \Quot_X(E,n)$ parameterising quotients that are entirely supported at a single (fixed) point in $X$.

Our first main result (proved in Theorem \ref{prop:Power_Structure_Formula}) is the following.

\begin{thm}\label{thmA}
There is an identity
\[
\mathsf Z_E(t) = \left( \sum_{n\geq 0}\mathsf P_{r,n}t^n\right)^{[X]}.
\]
\end{thm}

Since the punctual Quot scheme only depends on $r$, $n$ and $\dim X$, it follows that the same holds true for the motive of $\Quot_X(E,n)$.
Note that this was proved for $r=1$ (the Hilbert scheme case) by Gusein-Zade, Luengo and Melle-Hern\'{a}ndez \cite{GLMHilb}.

\smallbreak
Our second main result is of \emph{relative} nature and concerns $X = \A^d$. The Quot-to-Chow morphism
\[
\Quot_X(E,n) \to \Sym^nX
\]
sends a quotient $E\onto Q$ to the support of $Q$, viewed as a zero-cycle with multiplicities. We consider the relative motive
\[
\mathsf Z^{\rel}(\A^d,r) = \sum_{n\geq 0}\,\bigl[\Quot_{\A^d}(\O^{\oplus r},n) \to \Sym^n \A^d \bigr]\,\in \,K_0(\Var_{\Sym \A^d})
\]
over the symmetric product of $\A^d$. We define classes $\Omega_{r,n} \in K_0(\Var_{\C})$ by \[
\sum_{n\geq 0}\mathsf P_{r,n}t^n = \Exp\left(\sum_{n > 0}\Omega_{r,n}t^n\right)
\]
where $\Exp$ is the \emph{motivic exponential} (see Section \ref{sec:Motivic_Exp}) induced by the lambda ring structure on $K_0(\Var_{\C})$. For $\A^d$, we refine Theorem \ref{thmA}  by showing (see Theorem \ref{thm:Relative_Exp}) that $\mathsf Z^{\rel}(\A^d,r)$ is generated on the small diagonal by the absolute motives $\Omega_{r,n}$.

\begin{thm}\label{thmB}
There is an identity
\[
\mathsf Z^{\rel}(\A^d,r) = \Exp_{\cup}\left(\sum_{n>0} \Omega_{r,n} \boxtimes \bigl[\A^d\xrightarrow{\Delta_n}\Sym^n\A^d\bigr] \right).
\]
\end{thm}

See \cite{DavisonR,Quot19} for analogues of this result in the context of motivic Donaldson--Thomas theory and \cite{BBS} for the calculation of the (absolute) \emph{virtual motive} of $\Hilb^n(\A^3)$.

\smallbreak
Finally, our last result (see Section \ref{sec:quot_curve_motivic}) is the full ``solution'' of the motivic theory of the Quot scheme of a smooth curve, which can be summarised by the identities
\[
\Omega_{r,n} = 
\begin{cases}
    [\P^{r-1}] & \textrm{if }n=1 \\
    0& \textrm{if }n>1.
\end{cases}
\]
\begin{thm}\label{thm_C}
If $E$ is a locally free sheaf on a smooth curve $C$, there is an identity
\[
\mathsf Z_E(t) = \Exp\left(\bigl[C\times \P^{r-1}\bigr]t\right).
\]
Moreover, in $K_0(\Var_{\Sym \A^1})$ there is an identity
\[
\mathsf Z^{\rel}(\A^1,r) = \Exp_\cup\left(\bigl[\P^{r-1}\bigr] \boxtimes \bigl[ \A^1 \xrightarrow{\id} \A^1\bigr] \right).
\]
\end{thm}

We use the first relation to compute the Hodge--Deligne polynomial of the smooth space $\Quot_C(E,n)$ for a proper curve $C$ (Proposition \ref{prop:Hodge}). We stress that the formula for $\mathsf Z_E$ in the proper case was already implicit in the calculation of \cite[Prop.~4.5]{BFP19}. 

In Section \ref{sec:sufaces} we discuss the case $r=1$ on a surface, where we find $\Omega_{1,n} = \L^{n-1}$ according to G\"{o}ttsche's formula \cite{LG1}. Finally, we conclude by proposing a geometric open problem related to punctual Quot schemes on curves.

\smallbreak
We work over the field of complex numbers throughout.

\subsection*{Acknowledgements}
We wish to thank Massimo Bagnarol, Barbara Fantechi and Fabio Perroni for helpful discussions. We thank the anonymous referee for making several suggestions that helped improving the text. Finally, we thank SISSA for the excellent working conditions.

\section{Motivic preliminaries}

In this section we recall a few motivic constructions that will be needed later. Most of this material is a simplified version of \cite[Section 1]{DavisonR}, adapted to suit the purposes of this paper.

\subsection{The Grothendieck ring of varieties}
Fix a complex scheme $S$ locally of finite type over $\C$. The Grothendieck ring of $S$-varieties
\[
K_0(\Var_S)
\]
is the free abelian group generated by isomorphism classes $[X\to S]$ of $S$-varieties modulo the \emph{scissor relations}, namely the identities
\[
\bigl[X\xrightarrow{f} S\bigr] \,=\,\bigl[Y\xrightarrow{f|_{Y}} S\bigr] + \bigl[X\setminus Y\xrightarrow{f|_{X\setminus Y}} S\bigr]
\]
imposed whenever $Y\subset X$ is a closed $S$-subvariety of $X$. The ring structure is given on generators by fibre product over $S$,
\be\label{Product_KVar}
[X\to S]\cdot [Y\to S] = [X\times_SY\to S].
\ee
The element
\[
\L = [\A^1\times_{\C}S \to S]\,\in\,K_0(\Var_S)
\]
is called the \emph{Lefschetz motive} (over $S$).
If $S'$ is another complex scheme, there is an external product
\be\label{External_product}
K_0(\Var_S) \times K_0(\Var_{S'}) \xrightarrow{\boxtimes} K_0(\Var_{S\times S'})
\ee
defined on generators by sending $([f\colon X\to S],[g\colon X'\to S'])\mapsto [f\times g\colon X\times X' \to S\times S']$.

A morphism $f\colon S\to T$ induces a ring homomorphism $f^\ast\colon K_0(\Var_T) \to K_0(\Var_S)$ by base change and a $K_0(\Var_T)$-linear map $f_!\colon K_0(\Var_S) \to K_0(\Var_T)$ defined on generators by composition with $f$.

\begin{definition}
We denote by $S_0(\Var_S)$ the semigroup of \emph{effective} motives, i.e.~the semigroup generated by isomorphism classes $[X\to S]$ of complex quasi-projective $S$-varieties modulo the scissor relations. The product \eqref{Product_KVar} turns $S_0(\Var_S)$ into a semiring. There is a natural semiring map $S_0(\Var_{S}) \to K_0(\Var_S)$, and we say that $\alpha \in K_0(\Var_S)$ is effective if it lies in the image of this map.
\end{definition}

\subsection{Equivariant motives and the quotient map}\label{sec:Quotient_Map}
Recall that if $S$ is a scheme with a \emph{good} action by a finite group $G$ (i.e.~an action such that every point of $S$ has an affine $G$-invariant open neighborhood), the quotient $S/G$ exists as a scheme. For instance, finite group actions on quasi-projective varieties are good.

\begin{definition}\label{Equivariant_K_Group}
Let $G$ be a finite group, $S$ a scheme with good $G$-action. We denote by $\widetilde{K}_0^{G}(\Var_S)$ the free abelian group generated by isomorphism classes $[X\to S]$ of $G$-equivariant $S$-varieties
with good action, modulo the $G$-equivariant scissor relations. We denote by $K_0^G(\Var_S)$ the quotient of $\widetilde{K}_0^{G}(\Var_S)$ by the relations 
\[
[V\to X\to S] = [\A^r_X \to S],
\]
where $V\to X$ is a $G$-equivariant vector bundle of rank $r$ over a $G$-equivariant $S$-variety $X$.
\end{definition}

There is a natural ring structure on $\widetilde{K}_0^{G}(\Var_S)$, where the product of two classes $[X\to S]$ and $[Y\to S]$ is given by taking the diagonal action on $X\times_SY$. The structures $f^*$, $f_!$ and $\boxtimes$ naturally extend to the equivariant setting, along with their basic compatibilities. For instance, if $f\colon S\to T$ (resp.~$g\colon S'\to T'$) is a $G$-equivariant (resp.~$G'$-equivariant) map, and $u$, $v$ are equivariant motives over $S$, $S'$, then
\be\label{PF_Box_Product}
(f\times g)_!(u\boxtimes v) = f_!u \boxtimes g_!v
\ee
in the $(G\times G')$-equivariant $K$-group over $T\times T'$.

One can define a $K_0(\Var_{S/G})$-linear map (cf.~\cite[Lemma 1.5]{DavisonR})
\be\label{map:quot1map}
\pi_G\colon \widetilde{K}_0^{G}(\Var_S) \to K_0(\Var_{S/G})
\ee
given on generators by taking the orbit space,
\[
\pi_G [X\to S] = [X/G \to S/G].
\]
This map does not always extend to $K_0^G(\Var_S)$. It does when $G$ acts freely on $S$, by \cite[Lemma 3.2]{Bittner05}.

\subsection{Lambda ring structures}\label{sec:lambda_rings}

Let $n>0$ be an integer, and let $\mathfrak{S}_n$ be the symmetric group of $n$ elements. By \cite[Lemma~1.6]{DavisonR}, namely the relative version of \cite[Lemma~2.4]{BBS}, there exist ``$n$-th power'' maps 
\be\label{powermap}
(\,\cdot\,)^{\otimes n}\colon K_0(\Var_S)\to  \widetilde{K}_0^{\,\mathfrak S_n}(\Var_{S^n})
\ee
where $S^n = S\times \cdots\times S$ is endowed with the natural $\mathfrak{S}_n$-action. The power map takes $[f\colon X\to S]$ to the class of the equivariant function $f^n\colon X^n\to S^n$.
For $A\in K_0(\Var_S)$, consider the classes 
\[
\pi_{\mathfrak{S}_n}(A^{\otimes n}) \in K_0(\Var_{S^n/\mathfrak S_n}).
\]
The \textit{lambda ring} operations on $K_0(\Var_{\C})$ are defined by 
\[
A\mapsto \sigma^n(A) = \pi_{\mathfrak{S}_n}(A^{\otimes n}) \in K_0(\Var_{\C})
\]
for effective classes $A \in K_0(\Var_{\C})$, and then taking the unique extension to a lambda ring structure on $K_0(\Var_{\C})$, determined by the relation
\begin{equation}
\label{lambda_rel}
\sum_{i=0}^n\sigma^i([X]-[Y])\sigma^{n-i}[Y]=\sigma^n[X].
\end{equation}

If $S$ comes with a commutative associative map $\nu\colon S\times S\to S$, we likewise define 
\[
\sigma_{\nu}^n(A)=\overline{\nu}_!\pi_{\mathfrak{S}_n}(A^{\otimes n}) \in K_0(\Var_S)
\]
on effective classes $A = [X\to S]$, where $\overline{\nu}$ is the map $S^n/\mathfrak{S}_n\to S$. One then uses the analogue of the relation \eqref{lambda_rel} to find a unique set of lambda ring operators $\sigma_{\nu}^n$ restricting to the previous identity on effective motives. 

As a special case, one can consider $(S,\nu)=(\mathbb N,+)$, viewed as a symmetric monoid in the category of schemes. We obtain lambda operations $\sigma^n = \sigma^n_+$ on $K_0(\Var_{\C})\llbracket t\rrbracket$ via the isomorphism
\be\label{Iso_Power_series}
K_0(\Var_{\C})\llbracket t\rrbracket \,\widetilde{\to}\,K_0(\Var_{\mathbb N})
\ee
defined by sending $\sum_{n\geq 0}[Y_n]t^n \mapsto \bigl[\coprod_{n\in \mathbb N}Y_n\to \set{n}\bigr]$.

\subsection{Power structures}\label{sec:Power_structures}

The main references for power structures are \cite{GLMps,GLMHilb}.

\begin{definition}[{\cite{GLMps}}]\label{def:power_structure}
A \emph{power structure} on a (semi)ring $R$ is a map 
\begin{align*}
(1+tR\llbracket t\rrbracket)\times R&\to 1+tR\llbracket t\rrbracket\\
(A(t),m)&\mapsto A(t)^m
\end{align*}
satisfying the following conditions: 
\begin{enumerate}
    \item $A(t)^0=1$,
    \item $A(t)^1=A(t)$, 
    \item $(A(t)\cdot B(t))^m=A(t)^m\cdot B(t)^m$, 
    \item $A(t)^{m+m'}=A(t)^m\cdot A(t)^{m'}$, 
    \item $A(t)^{mm'}=(A(t)^m)^{m'}$, 
    \item $(1+t)^m=1+mt+O(t^2)$, 
    \item $A(t)^m\big{|}_{t\to t^e}=A(t^e)^m$.
\end{enumerate}  
\end{definition}

Throughout we use the following:
\begin{notation}
Partitions $\alpha\vdash n$ are written as $\alpha=(1^{\alpha_1}\cdots i^{\alpha_i}\cdots s^{\alpha_s})$, meaning that there are $\alpha_i$ parts of size $i$. In particular we recover $n = \sum_ii\alpha_i$.
The \emph{automorphism group} of $\alpha$ is the product of symmetric groups $G_\alpha=\prod_i\mathfrak S_{\alpha_i}$. 
\end{notation}

\begin{example}
If $R = \Z$, $A(t) = 1+\sum_{n>0}A_nt^n \in \Z\llbracket t \rrbracket$ and $m\in \mathbb N$, the known formula \cite[p.~40]{RPStanley}
\[
A(t)^m=1+\sum_{n\geq 0}\sum_{\alpha\vdash n}\,\,\left(\prod_{i=0}^{||\alpha||-1}(m-i)\cdot \frac{\prod_iA_i^{\alpha_i}}{\prod_i\alpha_i!}\right)t^{n}
\]
defines a power structure on $\Z$, where we have set $||\alpha|| = \sum_i\alpha_i$.
\end{example}

Gusein-Zade, Luengo and Melle-Hern{\'a}ndez have proved \cite[Thm.~2]{GLMps} that there is a unique power structure 
\[
(A(t),m) \mapsto A(t)^{m}
\]
on $K_0(\Var_{\C})$ extending the one defined in \emph{loc.~cit.}~on the semiring $S_0(\Var_{\C})$ of effective motives. The latter is given by the formula
\be\label{eqn:power_formula}
A(t)^{[X]}=1+\sum_{n\geq 0}\sum_{\alpha\vdash n}\pi_{G_\alpha}\Biggl(\Biggl[\prod_i X^{\alpha_i}\setminus \Delta\Biggr]\cdot \prod_i A_i^{\otimes \alpha_i}\Biggr)t^{n}.
\ee
Here, $\Delta\subset \prod_i X^{\alpha_i}$  is the ``big diagonal'' (the locus in the product where at least two entries are equal), and the product in big round brackets 
is a $G_\alpha$-equivariant motive in $\widetilde{K}_0^{G_\alpha}(\Var_{\C})$, thanks to the power map \eqref{powermap}.

\begin{remark}
We will not encounter non-effective coefficients in this paper, so we will have direct access to Formula \eqref{eqn:power_formula}.
\end{remark}

\subsection{Motivic exponential}\label{sec:Motivic_Exp}
The \emph{motivic exponential} is a group isomorphism
\[
\Exp\colon tK_0(\Var_{\C})\llbracket t\rrbracket\,\widetilde{\to}\,1+tK_0(\Var_{\C})\llbracket t\rrbracket,
\]
converting sums into products and preserving effectiveness. If $A = \sum_{n>0}A_n t^n$ is an effective power series, one has by definition
\[
\Exp\left(\sum_{n>0}A_nt^n \right) = \prod_{n>0}\,\bigl(1-t^n\bigr)^{-A_n},
\]
and if $A$ and $B$ are effective, one sets
\be\label{General_Exp}
\Exp(A-B) = \prod_{n>0}\,\bigl(1-t^n\bigr)^{-A_n}\cdot \left(\prod_{n>0}\,\bigl(1-t^n\bigr)^{-B_n}\right)^{-1}.
\ee
More generally, if $(S, \nu\colon S\times S\rightarrow S)$ is a commutative monoid in the category of schemes, with a submonoid $S_+ \subset S$ such that the induced map $\coprod_{n\geq 1}S_+^{\times n}\rightarrow S$ is of finite type, we similarly define
\[
\Exp_{\nu}(A)=\sum_{n\geq 0} \sigma_\nu^n(A)
\]
on effective classes, and for $A$ and $B$ two effective classes, we define $\Exp_\nu(A-B)$ by the analogue of \eqref{General_Exp}, i.e.~by $\Exp_\nu(A)\cdot \Exp_\nu(B)^{-1}$.

\subsection{Motives over symmetric products}\label{sec:symm_products_cup}

The machinery described so far will be applied to the following situation. For a variety $X$, we will consider $(\Sym(X),\cup)$, where
\[
\Sym(X) = \coprod_{n\geq 0}\Sym^n(X)
\]
can be viewed as a monoid via the morphism
\[
\Sym(X)\times\Sym(X)\xrightarrow{\cup} \Sym(X)
\]
sending two zero-cycles (with multiplicities) on $X$ to their union. The submonoid $\Sym(X)_+ = \coprod_{n>0}\Sym^n(X)$ allows one to construct the map $\Exp_\cup$ as in Section \ref{sec:Motivic_Exp}.

In order to recover a formal power series in $K_0(\Var_{\C})\llbracket t \rrbracket$ from a relative motive over $\Sym(X)$, we consider the operation 
\be\label{Power_Series_Sym}
\#_!\left(\sum_{n\geq 0}\,\bigl[Y_n\to \Sym^n X\bigr]\right)=\sum_{n\geq 0}\, [Y_n]t^n.
\ee
In other words we take the direct image along the ``tautological'' map $\#\colon \Sym(X)\rightarrow\mathbb{N}$ which collapses $\Sym^n(X)$ onto the point $n$. In the right hand side of \eqref{Power_Series_Sym}, we use the isomorphism \eqref{Iso_Power_series} to identify relative motives over $\mathbb N$ and formal power series with coefficients in $K_0(\Var_{\C})$.

The following result, a special case of \cite[Prop.~1.12]{DavisonR}, will be needed in the proof of Theorem \ref{thm:Relative_Exp}.

\begin{lemma}\label{Preparation_Exp}
Let $U$ be a variety and let $\Delta_n\colon U \to \Sym^nU$ be the small diagonal. Let $A=\sum_{n>0} A_n$ be an effective motive over $\mathbb N_{>0}$ and set $B = \Exp(A) = 1+\sum_{n>0} B_n$. Define
\[
\mathsf Z = \sum_{n\geq 0}\sum_{\alpha\vdash n} \cup_!\pi_{G_\alpha}j_\alpha^\ast
\left(\underset{i\mid \alpha_i\neq 0}{\Boxtimes}\Delta_{i\,!}\bigl(\bigl[U\xrightarrow{\id}U\bigr]\boxtimes
B_i\bigr)^{\otimes \alpha_i}\right) \in K_0(\Var_{\Sym U}),
\]
where $j_\alpha$ is the $G_\alpha$-equivariant open immersion $\prod_i \Sym^i(U)^{\alpha_i}\setminus \Delta \into \prod_i \Sym^i(U)^{\alpha_i}$. Then there is an identity
\[
\mathsf Z = \Exp_\cup\left(\sum_{n>0} A_n \boxtimes \bigl[ U\xrightarrow{\Delta_n} \Sym^nU\bigr]\right)
\]
 Moreover,
\[
\#_! \Exp_\cup\left(\sum_{n>0} A_n \boxtimes \bigl[ U\xrightarrow{\Delta_n} \Sym^nU\bigr]\right) = B^{[U]}\,\in\,K_0(\Var_{\C})\llbracket t \rrbracket.
\]
\end{lemma}

We briefly explain how to read the right hand side of the first equation of the lemma. First of all, we view $\cup$ as a map $\Sym(U)^b \to \Sym(U)$ for any $b>0$. The map $\pi_{G_\alpha}$ appearing in the definition of $\mathsf Z$ sends a $G_\alpha$-equivariant relative motive over $\prod_i \Sym^i(U)^{\alpha_i}\setminus \Delta$ to a relative motive over $\prod_i \Sym^{i\alpha_i}(U)\setminus \Delta$, therefore we can apply the direct image $\cup_!$ to get a relative motive over $\Sym^nU$, where $n = \sum_i i\alpha_i$.

\section{The motive of the Quot scheme}

\subsection{Main characters}
Let $X$ be a smooth quasi-projective variety of dimension $d$. Let $E$ be a rank $r$ locally free sheaf on $X$. For a given integer $n\geq 0$, the Quot scheme
\[
\Quot_X(E,n)
\]
parameterises quotients $E\onto Q$ such that
\[
\dim\, (\Supp Q) = 0,\quad \chi(Q) = n.
\]
The Quot-to-Chow map 
\[
\sigma_n\colon \Quot_X(E,n) \to \Sym^n X
\]
constructed in \cite[Section 6]{Grothendieck_Quot} (see also \cite[Cor.~$7.15$]{Rydh1} for a modern treatment) takes a quotient $E\onto Q$ to the zero-cycle (with multiplicities) determined by the set-theoretic support of $Q$. We define the \emph{punctual Quot scheme} to be the preimage
\[
\Quot_X(E,n)_p = \sigma_n^{-1}(n\cdot p)
\]
of the cycle $n\cdot p \in \Sym^nX$, where $p\in X$ is a point. This is easily seen to only depend on a formal neighborhood of $p\in X$ (but not on $p$, $X$ or $E$).
In particular, one has isomorphisms
\be\label{Punctual_Isomorphisms}
\Quot_X(E,n)_p \cong \Quot_{X}(\O_X^{\oplus r},n)_p \cong \Quot_{\A^d}(\O_{\A^d}^{\oplus r},n)_0
\ee
where $0$ is the origin in $\A^d$. This scheme will be denoted $P_{r,n}$ from now on, and
\[
\mathsf P_{r,n} = \bigl[P_{r,n}\bigr] \,\in \, K_0(\Var_{\C})
\]
will denote its motive. 

We pause for a second to explain how to prove the second isomorphism in \eqref{Punctual_Isomorphisms}. Using smoothness of $X$, we can fix \'etale coordinates around $p \in X$. This means we can find a pair $(U,\varphi)$ where $p \in U \subset X$ is an open neighborhood and $\varphi\colon U\to \mathbb A^d$ is an \'etale map such that $\varphi(p) = 0 \in \mathbb A^d$. As in the proof of \cite[Lemma A.1]{BR18}, we can further shrink $U$ until $U \cap \varphi^{-1}(0)$ is the single point $p$. Then, we consider the open subscheme $W \subset \Quot_{U}(\mathscr O_U^{\oplus r},n) \subset \Quot_X(\mathscr O_X^{\oplus r},n)$ consisting of quotients $\mathscr O_U^{\oplus r} \onto Q$ such that $\varphi|_{\Supp Q}$ is injective. Note that $W$ contains $\Quot_U(\mathscr O_U^{\oplus r},n)_p = \Quot_X(\mathscr O_X^{\oplus r},n)_p$ as a closed subscheme. By \cite[Proposition A.3]{BR18}, sending
\[
\left(\mathscr O_U^{\oplus r} \onto Q\right) \,\,\,\mapsto \,\,\,\left(\mathscr O_{\mathbb A^d}^{\oplus r} \to \varphi_\ast\varphi^\ast \mathscr O_{\mathbb A^d}^{\oplus r} = \varphi_\ast \mathscr O_U^{\oplus r} \onto \varphi_\ast Q \right)
\]
defines an \'etale morphism $\Phi\colon W \to \Quot_{\mathbb A^d}(\mathscr O_{\mathbb A^d}^{\oplus r},n)$. Its restriction 
\begin{equation}\label{restriction_to_punctual}
\Phi^{-1}\left(\Quot_{\mathbb A^d}(\mathscr O_{\mathbb A^d}^{\oplus r},n)_0 \right) \to \Quot_{\mathbb A^d}(\mathscr O_{\mathbb A^d}^{\oplus r},n)_0
\end{equation}
to the punctual Quot scheme of $\mathbb A^d$ is \'etale and bijective, hence an isomorphism. For surjectivity, use that $p$ is the only point in $U\cap \varphi^{-1}(0)$, and for injectivity use that $\varphi|_U$ is an immersion around $p$, so that $\varphi^\ast \varphi_\ast Q \,\widetilde{\to}\,Q$ is an isomorphism for all $Q$ supported entirely at $p$. Finally, again by our choice of $U$, the source of the morphism \eqref{restriction_to_punctual} is naturally identified with $\Quot_U(\mathscr O_U^{\oplus r},n)_p$.

\begin{remark}
The \emph{punctual motives} $\mathsf P_{r,n}$ clearly depend on the dimension $d = \dim X$, but we omit $d$ from the notation. 
\end{remark}

\begin{example}
On a curve (i.e.~if $d = 1$), by \cite[Prop.~2.6]{BFP19} we have
\be\label{Punctual_on_Curve}
\mathsf P_{r,1} = \bigl[\P^{r-1}\bigr].
\ee
\end{example}

\subsection{Absolute motives}
Let $X$ and $E$ be as in the previous section. Define the generating functions
\begin{align*}
    \mathsf P_r(t) &= \sum_{n\geq 0}\,\mathsf P_{r,n}t^n, \\
    \mathsf Z_E(t) &= \sum_{n\geq 0}\,\bigl[\Quot_X(E,n)\bigr] t^n
\end{align*}
in the power series ring $K_0(\Var_{\C})\llbracket t \rrbracket$.
The following result (namely Theorem \ref{thmA} from the Introduction) is the higher rank analogue of the corresponding statement for the Hilbert scheme of points \cite[Thm.~1]{GLMHilb}, obtained by setting $r=1$.

\begin{theorem}\label{prop:Power_Structure_Formula}
Let $X$ be a smooth quasi-projective variety. Let $E$ be a rank $r$ locally free sheaf on $X$.
There is an identity
\be\label{eqn:Z_Power_Structure}
\mathsf Z_E(t) = \mathsf P_r(t)^{[X]}.
\ee
\end{theorem}

\begin{proof}
For $\alpha$ a partition of $n$, let $\Sym^\alpha X\subset \Sym^nX$ be the locally closed subvariety parameterising zero-cycles whose support is distributed according to $\alpha$. We get a motivic decomposition
\be\label{eqn:Stratification}
\bigl[\Quot_X(E,n)\bigr] = \sum_{\alpha\vdash n}\,\bigl[\Quot_X(E,n)_{\alpha}\bigr],
\ee
where we have set $\Quot_X(E,n)_\alpha = \sigma_n^{-1}(\Sym^\alpha X)$. By standard arguments (see e.g.~\cite[Sec.~4]{BFHilb} and \cite[Sec.~3]{LocalDT}), one sees that the deepest stratum of the Quot-to-Chow map
\[
\sigma_{(n)}\colon \Quot_X(E,n)_{(n)} \to X
\]
is a Zariski locally trivial fibration with fibre $P_{r,n}$. This relies on the local case $X=\A^d$, where one has a global decomposition
\[
\Quot_{\A^d}(\O^{\oplus r},n)_{(n)} \cong \A^d \times P_{r,n}
\]
under which $\sigma_{(n)}$ is identified with the first projection. 

For a fixed partition $\alpha\vdash n$, let
\[
V_\alpha \into \prod_{i}\Quot_X(E,i)^{\alpha_i}
\]
be the open subscheme parameterising finite quotients with disjoint supports. By \cite[Prop.~A.3]{BR18} (but see also \cite[Lemma 4.10]{BFHilb} for the Hilbert scheme version), taking the union of points gives an \'etale map 
\[
u_\alpha\colon V_\alpha \to \Quot_X(E,n)
\]
and we let $U_\alpha$ denote its image. The stratum $\Quot_X(E,n)_\alpha$ sits inside $U_\alpha$ as a closed subscheme. We let the cartesian diagram
\be\label{diagram1}
\begin{tikzcd}[row sep = large]
Z_\alpha\MySymb{dr}\arrow[hook]{r}\arrow[swap]{d}{\widetilde u_\alpha} & V_\alpha \arrow{d}{u_\alpha}\\
\Quot_X(E,n)_\alpha \arrow[hook]{r} & U_\alpha
\end{tikzcd}
\ee
define the scheme $Z_\alpha$. The map $\widetilde u_\alpha$ is a finite \'etale cover with Galois group $G_\alpha$, in particular we have
\be\label{Quot_Quotient}
\Quot_X(E,n)_\alpha = Z_\alpha/G_\alpha.
\ee
In fact, $Z_\alpha$ can also be realised as the fibre product
\be\label{diagram2}
\begin{tikzcd}[row sep = large]
Z_\alpha\MySymb{dr}\arrow[hook]{r}{}\arrow[swap]{d}{f_\alpha} & \prod_i \Quot_X(E,i)_{(i)}^{\alpha_i}\arrow{d}\\
\prod_iX^{\alpha_i}\setminus \Delta \arrow[hook]{r}{} & \prod_iX^{\alpha_i}
\end{tikzcd}
\ee
where the bottom open immersion is the complement of the big diagonal and the map $f_\alpha$ is a $G_\alpha$-equivariant piecewise trivial fibration with fibre $ \prod_iP_{r,i}^{\alpha_i}$. This implies the identity
\[
\bigl[Z_\alpha\bigr] = \left[\prod_i X^{\alpha_i}\setminus \Delta\right] \cdot \prod_i\,\mathsf P_{r,i}^{{\otimes}\alpha_i}
\]
in $K_0^{G_\alpha}(\Var_{\C})$. 
Using \eqref{Quot_Quotient}, it follows that
\begin{align*}
    \bigl[\Quot_X(E,n)_\alpha\bigr] 
     \,&=\,\pi_{G_\alpha}\bigl[Z_\alpha\bigr] \\
     \,&=\,\pi_{G_\alpha}\left(\left[\prod_i X^{\alpha_i}\setminus \Delta\right] \cdot \prod_{i}\,\mathsf P_{r,i}^{{\otimes}\alpha_i}\right),
\end{align*}
where $\pi_{G_\alpha}\colon K_0^{G_\alpha}(\Var_{\C}) \to K_0(\Var_{\C})$ is the quotient map extending \eqref{map:quot1map}.
Since the classes $\mathsf P_{r,i}$ are effective, combining the decomposition \eqref{eqn:Stratification} with the power structure formula \eqref{eqn:power_formula} and summing over $n$ proves the result.
\end{proof}

The following is a generalisation of \cite[Thm.~4.1]{BFP19} to arbitrary varieties.

\begin{corollary}
The series $\mathsf Z_E(t)$ does not depend on $E$. In particular, the identity
\[
\bigl[\Quot_X(E,n)\bigr] = \bigl[\Quot_X(\O_X^{\oplus r},n)\bigr]
\]
holds in $K_0(\Var_{\C})$ for all locally free sheaves $E$ of rank $r$ on $X$.
\end{corollary}

\begin{definition}
Define absolute classes $\Omega_{r,n} \in K_0(\Var_{\C})$ via
\be\label{Omega_Classes_Definition}
\Exp\left(\sum_{n>0}\Omega_{r,n}t^n\right) = \mathsf P_r(t).
\ee
\end{definition}

\begin{remark}
In terms of the motivic exponential, we can rephrase Equation \eqref{eqn:Z_Power_Structure} as
\be\label{eqn:Relative_Motive_Exp}
\mathsf Z_E(t) = \Exp\left([X] \sum_{n>0} \Omega_{r,n}t^n \right).
\ee
It is then clear that to determine the series $\mathsf Z_E$ one has to compute the fully punctual classes $\Omega_{r,n}$. We will do this in the case of curves (for arbitrary $r$) in Section \ref{sec:quot_curve_motivic}, and for surfaces (only for $r=1$) in Section \ref{sec:sufaces}.
\end{remark}

\subsection{Relative motives}
Let $(X,E)$ be as in the previous sections. Consider the relative motive
\[
\mathsf Z_E^{\rel} =  \sum_{n\geq 0}\,\left[\Quot_X(E,n) \xrightarrow{\sigma_n} \Sym^n X\right]  \,\in\, K_0(\Var_{\Sym X}).
\]
In other words, $\mathsf Z_E^{\rel} = [\Quot_X(E) \to \Sym X]$, the class of $\Quot_X(E) = \coprod_{n} \Quot_X(E,n)$ over $\Sym X$.
Note that $\mathsf Z_E^{\rel}$ is a refinement of $\mathsf Z_E$, in the sense that
\[
\#_!\mathsf Z_E^{\rel} = \mathsf Z_E(t),
\]
where $\#_!$ is the operation introduced in \eqref{Power_Series_Sym}.

We simply write
\[
\mathsf Z^{\rel}(X,r) = \mathsf Z^{\rel}_{\O^{\oplus r}}
\]
when $E = \O^{\oplus r}$ is the trivial bundle over $X$.
We will show below (Theorem \ref{thm:Relative_Exp}) that the relative motive $\mathsf Z^{\rel}(\A^d,r) \in K_0(\Var_{\Sym \A^d})$ is generated under $\Exp_\cup$ by the motives $\Omega_{r,n}$ defined in \eqref{Omega_Classes_Definition}, extended on the small diagonal
\[
\A^d\xrightarrow{\Delta_n} \Sym^n \A^d.
\]

\begin{example}
Set $r=1$, $d=1$ (i.e.~we consider line bundles on curves). Then $\Quot_X(L,n) = \Hilb^nX = \Sym^nX$ for all line bundles $L$ on $X$, and
\[
\mathsf Z^{\rel}(X,1) = \mathsf Z_{\O_X}^{\rel} = \bigl[\Sym X \xrightarrow{\id}\Sym X\bigr] = \mathbb 1 \,\in\,K_0(\Var_{\Sym X}).
\]
Pushing this forward via $\#$ yields
\[
\mathsf Z_{\O_X}(t) = \sum_{n\geq 0}\,\bigl[\Sym^nX\bigr]t^n = \zeta_X(t),
\]
the Kapranov \emph{motivic zeta function} of the curve $X$. 
\end{example}

\begin{remark}
By definition of the power structure and of the motivic exponential, one has
\[
\zeta_Y(t) = (1-t)^{-[Y]} = \Exp([Y]t),
\]
for every variety $Y$. Moreover, the identities
\be\label{zeta_Shifted}
\zeta_Y(\L^st) = \zeta_{\A^s\times Y}(t) = \Exp(\L^s[Y]t)
\ee
hold in $K_0(\Var_{\C})\llbracket t \rrbracket$ for every $s \in \mathbb N$.
\end{remark}

\smallbreak
We now prove Theorem \ref{thmB} from the Introduction. 

Before we begin, let us observe that for a morphism of varieties $f\colon S\to T$  and an integer $n>0$, there is a commutative diagram
\be\label{PF_and_Power_Map}
\begin{tikzcd}[row sep = large]
K_0(\Var_S)\CommDiag{dr} \arrow{r}{f_!}\arrow[swap]{d}{(\,\cdot\,)^{\otimes n}} & K_0(\Var_T)\arrow{d}{(\,\cdot\,)^{\otimes n}} \\
\widetilde{K}_0^{\mathfrak S_n}(\Var_{S^n}) \arrow[swap]{r}{f^n_!} & \widetilde{K}_0^{\mathfrak S_n}(\Var_{T^n})
\end{tikzcd}
\ee
where $(\,\cdot\,)^{\otimes n}$ is the power map \eqref{powermap}.

\begin{theorem}\label{thm:Relative_Exp}
There is an identity
\[
\mathsf Z^{\rel}(\A^d,r) = \Exp_{\cup}\left(\sum_{n>0} \Omega_{r,n} \boxtimes \bigl[\A^d\xrightarrow{\Delta_n}\Sym^n\A^d\bigr] \right)\,\in\,K_0(\Var_{\Sym \A^d}).
\]
\end{theorem}

\begin{proof}
For a partition $\alpha\vdash n$, set $Q^n_\alpha = \Quot_{\A^d}(\O^{\oplus r},n)_{\alpha}$. One has a decomposition
\[
\mathsf Z^{\rel}(\A^d,r) =
\sum_{n\geq 0}\sum_{\alpha\vdash n} \,\bigl[ Q^n_\alpha \to \Sym^n \A^d\bigr].
\]
Let us consider the $G_\alpha$-equivariant cartesian diagram
\be\label{Complement_Big_Diagonal}
\begin{tikzcd}[row sep = large]
Z_{\alpha}\MySymb{dr}\arrow{d} \arrow[hook]{r} & \prod_i Q^{\alpha_i}_{(i)}\arrow{d} \\
\prod_i (\A^d)^{\alpha_i}\setminus \Delta \MySymb{dr} \arrow[hook]{r}{\iota_\alpha}\arrow[hook]{d}{\Delta}  & \prod_i (\A^d)^{\alpha_i} \arrow[hook]{d}{\Delta} \\
\prod_i \Sym^i(\A^d)^{\alpha_i}\setminus \Delta \arrow[hook]{r}{j_\alpha} & \prod_i \Sym^i(\A^d)^{\alpha_i}
\end{tikzcd}
\ee
where the top square is Diagram \eqref{diagram2}, the horizontal maps are open immersions (the complements of the big diagonals) and the vertical inclusions are products of small diagonals. We have a base change identity
\be\label{Base_Change_Motives}
j_\alpha^\ast \Delta_! = \Delta_!\iota_\alpha^\ast.
\ee
On the deepest stratum, we have a commutative diagram
\[
\begin{tikzcd}
Q^n_{(n)} \arrow[swap]{d}{\sigma_{(n)}} \arrow{r}{\sim} & \A^d \times P_{r,n} \arrow{dl}{\textrm{pr}_1} \\
\A^d &
\end{tikzcd}
\]
inducing an identity
\be\label{Fully_Punctual_Motives}
\bigl[Q^n_{(n)}\to \A^d\bigr] = \bigl[\A^d\xrightarrow{\id}\A^d\bigr] \boxtimes \mathsf P_{r,n}\,\in\, K_0(\Var_{\A^d}). 
\ee
For a general partition $\alpha$ of $n$, consider the equivariant motives
\[
\bigl[Q^i_{(i)}\to \A^d\bigr]^{\otimes \alpha_i}\,\in\,\widetilde K_0^{\mathfrak S_{\alpha_i}}\left(\Var_{(\A^d)^{\alpha_i}}\right).
\]
If $\iota_\alpha$ is as in Diagram \eqref{Complement_Big_Diagonal}, one has
\[
\Delta_!\left[Z_\alpha \to \prod_i\,(\A^d)^{\alpha_i}\setminus \Delta \right] = \Delta_!\iota_\alpha^\ast\left(\underset{i\lvert\alpha_i\neq 0}{\Boxtimes}\bigl[Q^i_{(i)}\to \A^d\bigr]^{\otimes \alpha_i}\right)\,\in\, 
\widetilde K_0^{G_\alpha}\left(\Var_{\prod_i \Sym^i(\A^d)^{\alpha_i}\setminus \Delta}\right).
\]
Applying the quotient map $\pi_{G_\alpha}$ to the last identity, followed by the pushforward along the union of points map, we obtain
\begin{align*}
\bigl[ Q^n_\alpha \to \Sym^n \A^d\bigr] &\,=\, \cup_!\pi_{G_\alpha}\Delta_!\iota_\alpha^\ast\left(\underset{i\lvert\alpha_i\neq 0}{\Boxtimes}\bigl[Q^i_{(i)}\to \A^d\bigr]^{\otimes \alpha_i}\right) & \\
&\,=\, \cup_!\pi_{G_\alpha}j_\alpha^\ast \Delta_!\left(\underset{i\lvert\alpha_i\neq 0}{\Boxtimes}\bigl[Q^i_{(i)}\to \A^d\bigr]^{\otimes \alpha_i}\right) & \textrm{by }\eqref{Base_Change_Motives}\\
&\,=\, \cup_!\pi_{G_\alpha}j_\alpha^\ast \left(\underset{i\lvert\alpha_i\neq 0}{\Boxtimes} (\Delta_i^{\alpha_i})_!\bigl[Q^i_{(i)}\to \A^d\bigr]^{\otimes \alpha_i} \right) & \textrm{by }\eqref{PF_Box_Product}\\
&\,=\,\cup_!\pi_{G_\alpha} j_\alpha^\ast\left(
\underset{i\lvert\alpha_i\neq 0}{\Boxtimes} \left(\Delta_{i\,!}\bigl[Q^i_{(i)}\to \A^d\bigr]\right)^{\otimes \alpha_i}\right) & \textrm{by }\eqref{PF_and_Power_Map} \\
&\,=\,\cup_!\pi_{G_\alpha} j_\alpha^\ast\left(
\underset{i\lvert\alpha_i\neq 0}{\Boxtimes} \Delta_{i\,!}\left(\bigl[\A^d\xrightarrow{\id}\A^d\bigr] \boxtimes \mathsf P_{r,i}\right)^{\otimes \alpha_i}\right) & \textrm{by } \eqref{Fully_Punctual_Motives}
\end{align*}
so that summing these classes over all partitions of integers and noting that $\Omega_{r,n}$ are effective (because $\mathsf P_{r,n}$ are effective) yields precisely
\[
\Exp_\cup\left(\sum_{n>0} \Omega_{r,n} \boxtimes \bigl[\A^d\xrightarrow{\Delta_n}\Sym^n\A^d\bigr] \right)
\]
by an application of Lemma \ref{Preparation_Exp}.
\end{proof}

\begin{remark}
By the last part of Lemma \ref{Preparation_Exp}, the theorem implies the formula
\[
\mathsf Z_{\O^{\oplus r}}(t) = \mathsf P_r(t)^{\L^d}
\]
of Theorem \ref{prop:Power_Structure_Formula} for $X=\A^d$.
\end{remark}

\subsection{Related work on more general Quot schemes}
The theory developed so far relies crucially on the locally free assumption on $E$. Indeed, the isomorphisms \eqref{Punctual_Isomorphisms} fail even if $E$ is, say, reflexive but not locally free. However, the geometry of the Quot scheme can be interesting also in the non-locally free case. For instance, the Quot scheme of finite quotients of the ideal sheaf $\mathscr I_C \subset \O_Y$ of a smooth curve in a $3$-fold $Y$ has been studied in \cite{LocalDT}, where essential local triviality statements on the Quot-to-Chow morphism were proved (see e.g.~Corollary 3.2 in \emph{loc.~cit.}). Moreover, in \cite[Thm.~2.1]{Ricolfi2018} it is proven that $\Quot_Y(\mathscr I_C,n)$ appears as the typical (scheme-theoretic) fibre of the Hilbert--Chow morphism $\Hilb(Y) \to \Chow(Y)$ in a neighborhood of the cycle of the smooth curve $C$. (This holds in all dimensions, not just $3$-folds.) This was used to prove the $C$-local DT/PT correspondence for Calabi--Yau $3$-folds \cite[Thm.~1.1]{Ricolfi2018}. 
The (virtual) motivic theory of $\Quot_Y(\mathscr I_C,n)$ was developed in \cite{DavisonR}. 

The enumerative geometry of $\Quot_X(E,n)$, for $E$ a sheaf of homological dimension at most one on a $3$-fold, was studied in \cite{BR18} and related to the local Pandharipande--Thomas theory of $X$. The Appendix in \emph{loc.~cit.}~develops the abstract theory comparing various Quot schemes of smooth quasi-projective varieties, and implicitly shows that the singularities of $\Quot_X(E,n)$ only depend on $n$ and $\dim X$. The (virtual) motivic theory in the locally free case for $3$-folds was developed in \cite{Quot19}, along with a construction of a virtual fundamental class on the Quot scheme.

\section{Calculations: curves and surfaces}\label{sec:curves}

In this section we compute the fully punctual motives
\[
\Omega_{r,n}
\]
in the case of curves, for all $r>0$ and $n>0$, and in the case of surfaces for $r=1$ and all $n>0$.

\subsection{The class of the Quot scheme on a curve}\label{sec:quot_curve_motivic}
We fix a locally free sheaf $E$ of rank $r$ on a smooth quasi-projective curve $C$.

\begin{lemma}\label{lemma_Omega_r_1}
On a curve, we have 
\[
\Omega_{r,1} = \bigl[\mathbb P^{r-1}\bigr].
\]
\end{lemma}

\begin{proof}
By the properties of the power structure, one has
\begin{align*}
    \mathsf P_r(t) &\,=\,\prod_{n\geq 1}\,(1-t^n)^{-\Omega_{r,n}} \\
    &\,=\,\prod_{n\geq 1}\,(1-t)^{-\Omega_{r,n}}\big|_{t\to t^n}\\
    &\,=\,\prod_{n\geq 1}\,(1+\Omega_{r,n}t+\cdots )\big|_{t\to t^n}\\
    &\,=\,\prod_{n\geq 1}\,(1+\Omega_{r,n}t^n+\cdots ),
\end{align*}
which immediately implies
\[
\Omega_{r,1} = \mathsf P_{r,1}.
\]
On the other hand, the equality $\mathsf P_{r,1} = [\P^{r-1}]$ holds by \eqref{Punctual_on_Curve}.
\end{proof}

We now reformulate (and generalise to the quasi-projective case) the main formula proved in \cite[Prop.~4.5]{BFP19}. The following is Theorem \ref{thm_C} from the Introduction.

\begin{theorem}\label{prop_BFP}
There is an identity 
\be\label{eqn:Absolute_Motive_Exp}
\mathsf Z_E(t) = \Exp\left(\left[C\times \P^{r-1}\right]t \right)
\ee
in $K_0(\Var_{\C})\llbracket t \rrbracket$.
Moreover, in $K_0(\Var_{\Sym \A^1})$ there is an identity
\[
\mathsf Z^{\rel}(\A^1,r) = \Exp_\cup\left(\bigl[\P^{r-1}\bigr] \boxtimes \bigl[ \A^1 \xrightarrow{\id} \A^1\bigr] \right).
\]
\end{theorem}

\begin{proof}
By  \cite[Prop.~4.5]{BFP19}, for projective $C$ one has
\be\label{eqn:BFP_Motive}
\bigl[\Quot_C(E,n)\bigr] = 
\sum_{n_1+\cdots+n_r = n} \bigl[\Sym^{n_1} C\bigr]\cdots \bigl[\Sym^{n_r} C\bigr]\cdot \L^{\sum_{i=0}^{r-1}(i-1)n_i}.
\ee
and it is clear that the generating function $\mathsf Z_E(t)$ of these motives can be expanded as a product of shifted motivic zeta functions. More precisely, one has 
\begin{align*}
\sum_{n \geq 0}\,\bigl[\Quot_C(E,n)\bigr] t^n &\,=\,\prod_{i=1}^r\zeta_C(\L^{i-1}t) \\
&\,=\,\prod_{i=1}^r\Exp([C]\L^{i-i}\cdot t) \\
&\,=\,\Exp\left([C]\sum_{i=1}^r\L^{i-i}\cdot \right) \\
&\,=\,\Exp\left(\bigl[C\times \P^{r-1}\bigr]t \right),
\end{align*}
where the second equality follows by \eqref{zeta_Shifted}.
So the statement is true when $C$ is projective. In this case, comparing \eqref{eqn:Absolute_Motive_Exp} with Equation \eqref{eqn:Relative_Motive_Exp} and using the injectivity of $\Exp$, we obtain the identities
\be\label{C_Omega_Vanishing}
[C] \cdot \Omega_{r,n}
=
\begin{cases}
    [C] \cdot [\P^{r-1}] & \textrm{if }n=1 \\
    0& \textrm{if }n>1.
\end{cases}
\ee
By Equation \eqref{eqn:Relative_Motive_Exp}, to prove the statement on an arbitrary $C$ it is enough to show that
\be\label{Omega_Classes_curve_1}
\Omega_{r,n}
=
\begin{cases}
    [\P^{r-1}] & \textrm{if }n=1 \\
    0& \textrm{if }n>1.
\end{cases}
\ee

By Lemma \ref{lemma_Omega_r_1}, we already know that $\Omega_{r,1} = [\P^{r-1}]$. Finally, the equation $\Omega_{r,n} = 0$ holds for $n>1$  because $\Omega_{r,n}$ is effective. Indeed, write $\Omega_{r,n} = [Y]$ for a variety $Y$, so that $0 = [C]\cdot \Omega_{r,n} = [C\times Y]$. But the class of a variety vanishes if and only if the variety is empty, and this happens if and only if $Y = \emptyset$.

To prove the last assertion, it is enough to combine Theorem \ref{thm:Relative_Exp} with the relations \eqref{Omega_Classes_curve_1}.
\end{proof}

\begin{remark}
The formula \eqref{eqn:BFP_Motive} is proved in \cite{BFP19} over a field $k$ of arbitrary charcateristic.
\end{remark}

\begin{remark}
By Equation \eqref{Omega_Classes_curve_1}, the generating function of the punctual motives can be computed as
\be\label{Punctual_On_Curves}
\mathsf P_r(t) = \Exp([\P^{r-1}]t) = \zeta_{\P^{r-1}}(t) = \prod_{i=0}^{r-1}\frac{1}{1-\L^it}.
\ee
\end{remark}

\subsection{The Hodge numbers of the Quot scheme on a curve}

The Hodge--Deligne polynomial (also called the E-polynomial) of a smooth complex projective variety $Y$ is given by 
\[
\mathsf E(Y;u,v) = \sum_{p,q}(-1)^{p+q}h^{p,q}(Y)u^pv^q
\]
where $h^{p,q}(Y) = \dim_{\C} H^q(Y,\Omega_Y^p)$ are the Hodge numbers of $Y$. 
For instance, one has
\be\label{eqn:E_Polynomials}
\begin{split}
    \mathsf E(\P^{r-1};u,v) &= \sum_{i=0}^{r-1}u^iv^i, \\
    \mathsf E(C;u,v) &= 1-gu-gv+uv,
\end{split}
\ee
where $C$ is a smooth projective curve of genus $g$.
Sending $[Y] \mapsto \mathsf E(Y;u,v)$ defines a motivic measure
\[
K_0(\Var_{\C}) \xrightarrow{\mathsf E} \Z[u,v]
\]
which is in fact a homomorphism of rings with power structure. The power structure on the polynomial ring $\Z[u,v]$ is determined by the formula
\[
(1-t)^{-f(u,v)} = \prod_{i,j} \,\bigl(1-u^iv^jt\bigr)^{-p_{ij}},
\]
where we have written $f(u,v) = \sum_{i,j}p_{ij}u^jv^j$ for integers $p_{ij}$.
This implies (cf.~\cite[Prop.~4]{GLMHilb}) the basic relation
\be\label{hom_power_structures}
\mathsf E\left((1-t)^{-[Y]}\right) = (1-t)^{-\mathsf E(Y;u,v)}.
\ee
Let $C$ be a smooth projective curve of genus $g$, and let $E$ be a rank $r$ locally free sheaf on $C$. We compute the generating function
\[
\mathsf E_r(C,t) = \sum_{n\geq 0} \mathsf E(\Quot_C(E,n);u,v) t^n.
\]
We already know this series does not depend on $E$.

\begin{prop}\label{prop:Hodge}
There is an identity
\be\label{eqn:Hodge_Deligne}
\mathsf E_r(C,t) = 
\prod_{i=0}^{r-1}\frac{(1-u^iv^{i+1}t)^g(1-u^{i+1}v^{i}t)^g}{(1-u^iv^it)(1-u^{i+1}v^{i+1}t)}
\ee
in the ring $\Z[u,v]\llbracket t\rrbracket$.
\end{prop}

\begin{proof}
We have
\begin{align*}
\mathsf E_r(C,t) &\,=\, \mathsf E\bigl(\Exp\bigl(\bigl[C\times \P^{r-1}\bigr]t\bigr)\bigr) & \textrm{by \eqref{eqn:Absolute_Motive_Exp}} \\
&\,=\,\mathsf E\left((1-t)^{-[C\times \P^{r-1}]}\right) & \textrm{by definition of Exp}\\
&\,=\,(1-t)^{-\mathsf E([C\times \P^{r-1}])} & \textrm{by \eqref{hom_power_structures}} \\
&\,=\,(1-t)^{-(1-gu-gv+uv)\sum_{i=0}^{r-1}u^iv^i}.
\end{align*}
We have used that $\mathsf E$ is a ring homomorphism and the identities \eqref{eqn:E_Polynomials} in the last step. The result now follows from direct computation and by definition of the power structure on $\Z[u,v]$.
\end{proof}

\begin{remark}
Setting $u=v$ in Formula \eqref{eqn:Hodge_Deligne} one recovers the generating function of (signed) Poincar\'e polynomials computed in \cite[Remark 4.6]{BFP19}, namely
\[
\sum_{n\geq 0} P(\Quot_C(F,n),-u)t^n = 
\prod_{i=0}^{r-1}\frac{(1-u^{2i+1}t)^{2g}}{(1-u^{2i}t)(1-u^{2i+2}t)}.
\]
\end{remark}

\subsection{The Hilbert scheme of points on a surface}\label{sec:sufaces}

Let $S$ be a smooth quasi-projective surface, and set $r=1$, so that $\Quot_S(L,n) = \Hilb^nS$ for every line bundle $L$.
We know by Formula \eqref{eqn:Relative_Motive_Exp} that 
\[
\mathsf Z_{\O_S}(t) = \Exp\left([S]\sum_{n>0} \Omega_{1,n}t^n \right).
\]
On the other hand, by G\"{o}ttsche's formula \cite{LG1},
\[
\mathsf Z_{\O_S}(t) = \Exp\left(\frac{[S]t}{1-\L t}\right) = \Exp\left([S]\sum_{n > 0}\L^{n-1} t^{n}\right).
\]
By the injectivity of $\Exp$, we conclude that on a surface $S$ one has
\[
(\Omega_{1,n} - \L^{n-1})[S] = 0.
\]
However, this relation holds universally for \emph{every} quasi-projective surface, in particular for $S=\P^2$ and $S=\A^1\times \P^1$. Therefore
\[
(\Omega_{1,n} - \L^{n-1})(1+\L+\L^2) = 0 = (\Omega_{1,n} - \L^{n-1})(\L+\L^2),
\]
showing that 
\be\label{Omega_Surface}
\Omega_{1,n} = \L^{n-1},\quad n>0.
\ee
In particular, we recover the known generating function of the motives of punctual Hilbert schemes, given by the formula
\be\label{Punctual_on_surfaces}
\mathsf P_1(t) = \sum_{n\geq 0}\,\bigl[\Hilb^n(\A^2)_0\bigr] t^n = \prod_{n\geq 1}\,\bigl(1-\L^{n-1}t^n\bigr)^{-1}.
\ee

Finally, we obtain the following relative statement.

\begin{theorem}
There is an identity
\[
\sum_{n\geq 0}\,\left[\Hilb^n\A^2 \xrightarrow{\sigma_n}\Sym^n\A^2 \right] = \Exp_\cup\left(\sum_{n>0}\L^{n-1}\boxtimes \bigl[\A^2\xrightarrow{\Delta_n}\Sym^n\A^2\bigr]\right)
\]
in $K_0(\Var_{\Sym \A^2})$.
\end{theorem}

\begin{proof}
Combine Theorem \ref{thm:Relative_Exp} with Equation \eqref{Omega_Surface}.
\end{proof}

\begin{remark}
The relation \eqref{Punctual_on_surfaces} was already proved in \cite{GLMps}, and it was exploited in \cite{MR18} to provide a motivic check of the classification of modules of length $3$ and $4$ over the polynomial ring $k[x,y]$.
\end{remark}

\section{A motivic-to-geometric open problem}
Let $C$ be a smooth quasi-projective curve. The punctual Quot scheme
\[
P_{r,n} \subset \Quot_C(\O^{\oplus r},n)
\]
parameterises quotients $\O^{\oplus r} \onto Q$ entirely supported at a single (fixed) point $p\in C$. As proved in \cite[Prop.~2.6]{BFP19}, one has
\[
P_{r,1} = \P^{r-1}.
\]
How can one describe $P_{r,n}$ for $n>1$? The relation
\[
\mathsf P_r(t) = \Exp\left(\bigl[\P^{r-1}\bigr] t \right) = \zeta_{\P^{r-1}}(t)
\]
established in Equation \eqref{Punctual_On_Curves} translates into the motivic identity
\be\label{Motivic_Punctual_Sym}
\bigl[P_{r,n}\bigr] = \bigl[\Sym^n\P^{r-1} \bigr] = \bigl[\Sym^n P_{r,1} \bigr].
\ee
It thus makes sense to ask the following:
\begin{question}
What is the geometric meaning of the relation \eqref{Motivic_Punctual_Sym}? Can one geometrically compare the schemes $P_{r,n}$ and $\Sym^n P_{r,1}$?
\end{question}

\bibliographystyle{amsplain-nodash}
\bibliography{bib}

\end{document}

%% file: macros.tex


\def\be{\begin{equation}}    
\def\ee{\end{equation}}
\def\bitem{\begin{itemize}}
\def\eitem{\end{itemize}}
\def\benum{\begin{enumerate}}
\def\eenum{\end{enumerate}}

\def\into{\hookrightarrow}
\def\onto{\twoheadrightarrow}

\def\Var{\mathrm{Var}}

\def\L{\mathbb L}
\def\A{\mathbb A}

\def\C{\mathbb C}

\def\P{\mathbb P}

\def\L{\mathbb{L}}

\def\Z{\mathbb Z}

\def\O{\mathscr O}

\DeclareMathOperator{\Boxtimes}{\mathlarger{\mathlarger{\boxtimes}}}

\DeclareMathOperator{\Hilb}{Hilb}
\DeclareMathOperator{\Quot}{Quot}
\DeclareMathOperator{\Chow}{Chow}

\DeclareMathOperator{\Sym}{Sym}

\DeclareMathOperator{\rel}{rel}

\DeclareMathOperator{\id}{id}

\DeclareMathOperator{\Supp}{Supp\,}

\DeclareMathOperator{\Exp}{Exp}

\theoremstyle{definition}

\newtheorem*{lemma*}{Lemma}
\newtheorem*{theorem*}{Theorem}
\newtheorem*{example*}{Example}
\newtheorem*{fact*}{Fact}
\newtheorem*{notation*}{Notation}
\newtheorem*{definition*}{Definition}
\newtheorem*{prop*}{Proposition}
\newtheorem*{remark*}{Remark}
\newtheorem*{corollary*}{Corollary}
\newtheorem*{conventions*}{Conventions}

\newtheorem{definition}{Definition}[section]
\newtheorem{example}[definition]{Example}

\newtheorem{question}[definition]{Question}
\newtheorem{notation}[definition]{Notation}
\newtheorem{remark}[definition]{Remark}

\newtheoremstyle{thm} 
        {3mm}
        {3mm}
        {\slshape}
        {0mm}
        {\bfseries}
        {.}
        {1mm}
        {}
\theoremstyle{thm}
\newtheorem{theorem}[definition]{Theorem}
\newtheorem{corollary}[definition]{Corollary}
\newtheorem{lemma}[definition]{Lemma}
\newtheorem{prop}[definition]{Proposition}
\newtheorem{thm}{Theorem}

\newtheoremstyle{ass} 
        {3mm}
        {3mm}
        {\rmfamily}
        {0mm}
        {\scshape}
        {.}
        {1mm}
        {}
\theoremstyle{ass}

\usepackage{tikz}
\usepackage{tikz-cd}
\usepackage{rotating}

\usetikzlibrary{matrix,shapes,arrows,decorations.pathmorphing}
\tikzset{commutative diagrams/arrow style=math font}
\tikzset{commutative diagrams/.cd,
mysymbol/.style={start anchor=center,end anchor=center,draw=none}}
\newcommand\CommDiag[2][\circlearrowleft]{%
  \arrow[mysymbol]{#2}[description]{#1}}

\newcommand\MySymb[2][\square]{%
  \arrow[mysymbol]{#2}[description]{#1}}
\tikzset{
shift up/.style={
to path={([yshift=#1]\tikztostart.east) -- ([yshift=#1]\tikztotarget.west) \tikztonodes}
}
}

\DeclareMathAlphabet{\mathpzc}{OT1}{pzc}{m}{it}